\documentclass[12pt, reqno]{amsart}
\usepackage{amssymb, amsmath, amsthm, hyperref, graphicx}

\newtheorem{thrm}{Theorem}
\newtheorem{lemma}{Lemma}

\begin{document}
\title[Necessary conditions for the extendibility of a 1-st order flex]{Necessary conditions 
for the extendibility of a first-order flex of a polyhedron to its flex}
\author{Victor Alexandrov}
\address{Sobolev Institute of Mathematics, Koptyug ave., 4, 
Novosibirsk, 630090, Russia and Department of Physics, 
Novosibirsk State University, Pirogov str., 2, Novosibirsk, 
630090, Russia}
\email{alex@math.nsc.ru}
\address{{}\hfill{November 2, 2019}}
\begin{abstract}
We derive fundamentally new equations that are satisfied by first-order flexes of a flexible polyhedron.
Moreover, we indicate two sources of such new equations.
These sources are the Dehn invariants and rigidity matrix.
The equations derived provide us with fundamentally new necessary conditions for the extendibility of 
a first-order flex of a polyhedron to its flex.
\par
\textit{Keywords}: Euclidean 3-space, flexible polyhedron, infinitesimal bending, Dehn invatiant, rigidity matrix.
\par
\textit{Mathematics subject classification (2010)}:  52C25. 	

\end{abstract}
\maketitle

\section{Introduction}\label{s1}

A polyhedron (more precisely, a compact polyhedral boundary-free surface in $\mathbb{R}^3$)
is called \textit{flexible} if its spatial form can be changed by a continuous deformation, 
in the course of which each face remains congruent to itself.
If some faces of a flexible polyhedron are not triangular, then one can triangulate them
in an arbitrary way and obtain a new flexible polyhedron with triangular faces only.
This is the reason why, in the theory of flexible polyhedra, it is customary to study
flexible polyhedra with triangular faces only.
Such a polyhedron is flexible, if and only if its spatial form can be
changed by a continuous motion of its vertices, during which the length of every edge remains
unaltered (so that the change in the spatial shape of the polyhedron results from the change of its dihedral angles only).

The first examples and properties of flexible polyhedra (more precisely, of flexible octahedra with self-intersections)
were discovered in the 19th century \cite{Br97}. 
Later flexible octahedra were studied in \cite{Be12} and \cite{Le67}.
However, the real flowering of the theory of flexible polyhedra began in the mid-1970s when
Robert Connelly constructed a sphere-homeomorphic self-intersection free flexible polyhedron \cite{Co77}.
Now we already know that any flexible polyhedron in Euclidean 3-space preserves the total mean curvature
(see \cite{Al85}), enclosed volume (there are especially many articles devoted to this issue, so we point out
several of them in chronological order: 
\cite{Sa95}, \cite{Sa96a}, \cite{Sa96b}, \cite{CSW97}, \cite{Sa98a}, \cite{Sa98b}, \cite{Ga14a}, and \cite{Ga14b}),
and Dehn invariants (see \cite{GI18}).
The reader can find more details about the theory of flexible polyhedra in the above mentioned
articles, as well as in the following review articles, which we list in chronological order:
\cite{Ku79}, \cite{Co80}, \cite{Sc04}, \cite{Sa11}, and \cite{Ga18}.

However, there are many intriguing open problems in the theory of flexible polyhedra.
One of them reads as follows: given a polyhedron, recognize whether it is flexible or not.
One approach to this problem is to analyze the infinitesimal flexes of the given polyhedron.

Recall that an infinitesimal flex (more precisely, an \textit{$n$th order flex} for some $n\geqslant 1$) 
of a polyhedron with triangular faces
is a continuous motion of the vertices (depending on some parameter $t$, which can be
interpreted as time, with the initial polyhedron corresponding to $t=0$)
that changes every edge length by a value which is $o(t^n)$ as $t\to 0$.
An $n$th order flex can be identified with the set of the derivatives of the position vectors of the vertices 
of orders $1,\dots, n$ of the initial polyhedron. 
An $n$th order flex satisfies some well-known equations.
Many geometers studied conditions under which an $n$th order flex can be extended to an $(n+1)$th order flex or to a flex.
The reader can find more details about the infinitesimal flexes of polyhedra and smooth surfaces,
for example, in the classic article \cite{Ef52}, review articles \cite{Co93}, \cite{IS95} and \cite{IMS8},
as well as in the book \cite{Kl14}.

In the present article, we derive fundamentally new equations that are satisfied by first-order flexes of a flexible polyhedron.
Moreover, we indicate two sources of such new equations.
Namely in Section \ref{s3}, we derive the desired equations using the fact that the Dehn invariants remain unaltered 
during the flex, a fact that was established for the first time by A.A. Gaifullin and L.S. Ignashchenko \cite{GI18}.
In Section \ref{s4}, we derive another set of desired equations using the notion of the rigidity matrix.
The equations derived provide us with fundamentally new necessary conditions for the extendibility of 
a first-order flex of a polyhedron to its flex.

Finally, we recall that, in \cite{Al98}, sufficient conditions are found for the extendibility of an $n$th order flex
of a polyhedron, $n\geqslant 1$, to its flex.
Those conditions are based on absolutely different ideas than those used in the present article.

\section{New equations for the velocity vectors of the vertices of a flexible polyhedron, generated by the Dehn invariants}\label{s2}

Let $P_0$ be an oriented polyhedron in $\mathbb{R}^3$ with non-degenerate triangular faces
(the latter means that no straight line contains all three vertices of a face).
Moreover, suppose that $P_0$ is a member of a continuous in $t$ family $\{ P_t \}_{t\in [0,1]}$
of polyhedra in $\mathbb {R}^3$ such that, for each $t\in (0,1]$, 
the corresponding (by continuity) edges of $P_t$ and $P_0$ have the same length.
If $P_t$ and $P_0$ are not congruent to each other for each $t\in (0,1]$ then the family
$\{P_t \}_{t\in [0,1]}$ is called a \textit{flex} of $P_0$ and $P_0$ is called \textit{flexible}.
$P_0$ is called \textit{rigid} if no flex of $P_0$ does exist. 

For each edge $\sigma$ of $P_0$ denote 

$\bullet$ by $\boldsymbol{x}_{\sigma}$ and $\boldsymbol{y}_{\sigma}$ 
the end-points of $\sigma$; 

$\bullet$ by $\boldsymbol{z}_{\sigma}'$ and $\boldsymbol{z}_{\sigma}''$ 
the vertices of $P_0$, such that the triples  
$\boldsymbol{x}_{\sigma}$, $\boldsymbol{y}_{\sigma}$, $\boldsymbol{z}_{\sigma}'$ and
$\boldsymbol{x}_{\sigma}$, $\boldsymbol{y}_{\sigma}$, $\boldsymbol{z}_{\sigma}''$ 
form faces of $P_0$;

$\bullet$ by $\ell_{\sigma}$, $\ell'_{\sigma}$, and $\ell''_{\sigma}$ the lengths of the vectors 
$\boldsymbol{y}_{\sigma}-\boldsymbol{x}_{\sigma}$, $\boldsymbol{z}_{\sigma}'-\boldsymbol{x}_{\sigma}$,
and $\boldsymbol{z}_{\sigma}''-\boldsymbol{x}_{\sigma}$ respectively;

$\bullet$ by $A'_{\sigma}$ and $A''_{\sigma}$ the doubled areas 
of the faces of $P_0$ with the vertices 
$\boldsymbol{x}_{\sigma}$, $\boldsymbol{y}_{\sigma}$, $\boldsymbol{z}_{\sigma}'$ and
$\boldsymbol{x}_{\sigma}$, $\boldsymbol{y}_{\sigma}$, $\boldsymbol{z}_{\sigma}''$ respectively
(equivalently, the lengths of the vectors 
$(\boldsymbol{y}_{\sigma}-\boldsymbol{x}_{\sigma})\times (\boldsymbol{z}_{\sigma}'-\boldsymbol{x}_{\sigma})$ 
and 
$(\boldsymbol{y}_{\sigma}-\boldsymbol{x}_{\sigma})\times (\boldsymbol{z}_{\sigma}''-\boldsymbol{x}_{\sigma})$
respectively);

$\bullet$ by $\varphi_{\sigma}$ the internal dihedral angle of the polyhedron $P_0$ at the edge $\sigma$.

Moreover, suppose that the notation for the vertices $\boldsymbol{x}_{\sigma}$ and $\boldsymbol{y}_{\sigma}$ 
is chosen so that the unit vectors 
\begin{equation}\label{eq1}
\boldsymbol{n}_{\sigma}'=\frac{(\boldsymbol{y}_{\sigma}-\boldsymbol{x}_{\sigma})\times 
(\boldsymbol{z}_{\sigma}'-\boldsymbol{x}_{\sigma})}{A'_{\sigma}}
\quad\text{and}\quad
\boldsymbol{n}_{\sigma}''=\frac{(\boldsymbol{z}_{\sigma}''-\boldsymbol{x}_{\sigma}) \times 
(\boldsymbol{y}_{\sigma}-\boldsymbol{x}_{\sigma})}{A''_{\sigma}}
\end{equation}
are the vectors of outward normals to the faces with the vertices
$\boldsymbol{x}_{\sigma}$, $\boldsymbol{y}_{\sigma}$, $\boldsymbol{z}_{\sigma}'$ 
and $\boldsymbol{x}_{\sigma}$, $\boldsymbol{y}_{\sigma}$, $\boldsymbol{z}_{\sigma}''$ 
respectively (see Fig.~\ref{fig1}).
\begin{figure}
\begin{center}
\includegraphics[width=0.25\textwidth]{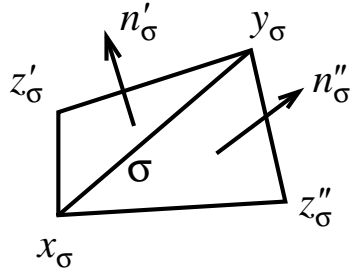}
\end{center}
\caption{Edge $\sigma$, the adjacent faces of $P_0$, and outward normals  
$\boldsymbol{n}_{\sigma}'$ and  $\boldsymbol{n}_{\sigma}''$ to them.}\label{fig1}
\end{figure}
Here and below $\boldsymbol{a}\times \boldsymbol{b}$ stands for the cross product of 
the vectors $\boldsymbol{a}$ and $\boldsymbol{b}$.

The above agreement that $\boldsymbol{n}_{\sigma}'$ and $\boldsymbol{n}_{\sigma}''$, 
defined by the formulas (\ref{eq1}), are outward normals to the corresponding faces of
$P_0$, does not limit the generality of reasoning.
In fact, in the case of violation of this agreement, it is sufficient to swap the notation for
$\boldsymbol{x}_{\sigma}$ and $\boldsymbol{y}_{\sigma}$.
After that, the vectors $\boldsymbol{n}_{\sigma}'$ and $\boldsymbol{n}_{\sigma}''$
will be replaced by $-\boldsymbol{n}_{\sigma}'$ and $-\boldsymbol{n}_{\sigma}''$,
respectively, and the agreement will be fulfilled.

For short, mainly in the proofs, we use the following notation:
$\boldsymbol{w}=\boldsymbol{y}_{\sigma}-\boldsymbol{x}_{\sigma}$, 
$\boldsymbol{w'}=\boldsymbol{z}_{\sigma}'-\boldsymbol{x}_{\sigma}$, and
$\boldsymbol{w''}=\boldsymbol{z}_{\sigma}''-\boldsymbol{x}_{\sigma}$.

\begin{lemma}\label{l1}
Under the above notation, for every edge $\sigma$ of $P_0$, the equalities hold true
\begin{eqnarray}
A'_{\sigma}A''_{\sigma}\cos\varphi_{\sigma} & = & 
\ell^2_{\sigma}(\boldsymbol{z}_{\sigma}'-\boldsymbol{x}_{\sigma})\cdot(\boldsymbol{z}_{\sigma}''-\boldsymbol{x}_{\sigma})
\nonumber \\ & & 
\hphantom{a}-\bigl[(\boldsymbol{z}_{\sigma}'-\boldsymbol{x}_{\sigma})\cdot (\boldsymbol{y}_{\sigma}-\boldsymbol{x}_{\sigma})\bigr]
\bigl[(\boldsymbol{z}_{\sigma}''-\boldsymbol{x}_{\sigma})\cdot (\boldsymbol{y}_{\sigma}-\boldsymbol{x}_{\sigma})\bigr],
\label{eq2} \\
A'_{\sigma}A''_{\sigma}\sin\varphi_{\sigma} & = &
- \ell_{\sigma}(\boldsymbol{y}_{\sigma}-\boldsymbol{x}_{\sigma})\cdot \bigl[(\boldsymbol{z}_{\sigma}'-\boldsymbol{x}_{\sigma})
\times (\boldsymbol{z}_{\sigma}''-\boldsymbol{x}_{\sigma})\bigr].  \label{eq3}
\end{eqnarray}
\textnormal{Here and below $\boldsymbol{a}\cdot \boldsymbol{b}$ 
stands for the scalar product of the vectors $\boldsymbol{a}$ and $\boldsymbol{b}$.}
\end{lemma}

\begin{proof}
The formula (\ref{eq2}) follows from the equality
$\cos\varphi_{\sigma} = - \boldsymbol{n}_{\sigma}'\cdot \boldsymbol{n}_{\sigma}''$ 
and the well-known identity
$(\boldsymbol{a}\times \boldsymbol{b})\cdot (\boldsymbol{c}\times \boldsymbol{d}) 
= (\boldsymbol{a}\cdot \boldsymbol{c})(\boldsymbol{b}\cdot \boldsymbol{d})
-(\boldsymbol{a}\cdot \boldsymbol{d})(\boldsymbol{b}\cdot \boldsymbol{c})$.
Indeed,
\begin{eqnarray*}
\cos\varphi_{\sigma} & = & -\frac{\boldsymbol{w}\times \boldsymbol{w'}}
{A'_{\sigma}}\times 
\frac{\boldsymbol{w''}\times \boldsymbol{w}}{A''_{\sigma}} \\
& = & -\frac{1}{A'_{\sigma}A''_{\sigma}}
\bigl[(\boldsymbol{w}\cdot \boldsymbol{w''})(\boldsymbol{w'}\cdot \boldsymbol{w})
-(\boldsymbol{w'}\cdot \boldsymbol{w''})(\boldsymbol{w}\cdot \boldsymbol{w})\bigr] \\
& = &
\frac{1}{A'_{\sigma}A''_{\sigma}}\bigl[\ell^2_{\sigma}
(\boldsymbol{w'}\cdot \boldsymbol{w''})-(\boldsymbol{w}\cdot \boldsymbol{w''})(\boldsymbol{w'}\cdot \boldsymbol{w})\bigr].
\end{eqnarray*}

The formula (\ref{eq3}) follows from the equality
\begin{equation*}
\boldsymbol{n}_{\sigma}'\times \boldsymbol{n}_{\sigma}'' = \frac{\sin\varphi_{\sigma}}{\ell_{\sigma}}\boldsymbol{w}
\end{equation*}
and the well-known identity
$(\boldsymbol{a}\times \boldsymbol{b})\times (\boldsymbol{c}\times \boldsymbol{d}) 
= \bigl[(\boldsymbol{a}\times \boldsymbol{b})\cdot \boldsymbol{d}\bigr]\boldsymbol{c} 
- \bigl[(\boldsymbol{a}\times \boldsymbol{b})\cdot \boldsymbol{c}\bigr]\boldsymbol{d}$.
Indeed,
\begin{eqnarray*}
\sin\varphi_{\sigma} & = & \frac{\ell_{\sigma}}{\boldsymbol{w}\cdot \boldsymbol{w}} \boldsymbol{w}
\cdot (\boldsymbol{n}_{\sigma}'\times \boldsymbol{n}_{\sigma}'') =
\frac{1}{\ell_{\sigma}} \boldsymbol{w}\cdot 
\biggl(
\frac{\boldsymbol{w}\times \boldsymbol{w'}}{A'_{\sigma}} \times 
\frac{\boldsymbol{w''}\times \boldsymbol{w}}{A''_{\sigma}} 
\biggr) \\
& = & \frac{1}{\ell_{\sigma}A'_{\sigma}A''_{\sigma}} \boldsymbol{w}
\cdot \bigl\{\bigl[(\boldsymbol{w}\times \boldsymbol{w'})\cdot \boldsymbol{w}\bigr]\boldsymbol{w''}-
\bigl[(\boldsymbol{w}\times \boldsymbol{w'})\cdot \boldsymbol{w''}\bigr]\boldsymbol{w}\bigr\} \\
& = & -\frac{\boldsymbol{w}\cdot \boldsymbol{w}}{\ell_{\sigma}A'_{\sigma}A''_{\sigma}}
(\boldsymbol{w'}\times \boldsymbol{w''})\cdot \boldsymbol{w} =
-\frac{\ell_{\sigma}}{A'_{\sigma}A''_{\sigma}}\boldsymbol{w}\cdot (\boldsymbol{w'}\times \boldsymbol{w''}). 
\end{eqnarray*}
\end{proof}

Since $\{P_t \}_{t\in [0,1]}$ is a flex of $P_0$,
for every $t\in [0,1]$ and every edge $\sigma$ of $P_0$, there is an edge $\sigma (t)$ of $P_t$ which corresponds to
$\sigma$ by continuity (in particular, $\sigma (0)=\sigma$).
Similarly, for the dihedral angle $\varphi_{\sigma}$ and vectors
$\boldsymbol{x}_{\sigma}$, $\boldsymbol{y}_{\sigma}$, $\boldsymbol{z}_{\sigma}'$, 
$\boldsymbol{z}_{\sigma}''$, $\boldsymbol{w}=\boldsymbol{y}_{\sigma}-\boldsymbol{x}_{\sigma}$, 
$\boldsymbol{w'}=\boldsymbol{z}_{\sigma}'-\boldsymbol{x}_{\sigma}$, and
$\boldsymbol{w''}=\boldsymbol{z}_{\sigma}''-\boldsymbol{x}_{\sigma}$, 
associated with the edge $\sigma$ of $P_0$, there are similar quantities which correspond
to the edge $\sigma (t)$ of $P_t$.
Denote them by
$\varphi_{\sigma}(t)$, $\boldsymbol{x}_{\sigma}(t)$, $\boldsymbol{y}_{\sigma}(t)$, $\boldsymbol{z}_{\sigma}'(t)$, 
$\boldsymbol{z}_{\sigma}''(t)$, $\boldsymbol{w}(t)=\boldsymbol{y}_{\sigma}(t)-\boldsymbol{x}_{\sigma}(t)$, 
$\boldsymbol{w'}(t)=\boldsymbol{z}_{\sigma}'(t)-\boldsymbol{x}_{\sigma}(t)$, and
$\boldsymbol{w''}(t)=\boldsymbol{z}_{\sigma}''(t)-\boldsymbol{x}_{\sigma}(t)$ 
respectively.
Denote the derivatives of these quantities with respect to $t$ calculated at $t=0$ by
$\Phi_{\sigma}$, $\boldsymbol{\xi}_{\sigma}$, $\boldsymbol{\eta}_{\sigma}$, $\boldsymbol{\zeta}_{\sigma}'$, 
$\boldsymbol{\zeta}_{\sigma}''$, $\boldsymbol{W}=\boldsymbol{\eta}_{\sigma}-\boldsymbol{\xi}_{\sigma}$, 
$\boldsymbol{W'}=\boldsymbol{\zeta}_{\sigma}'-\boldsymbol{\xi}_{\sigma}$, and 
$\boldsymbol{W''}=\boldsymbol{\zeta}_{\sigma}''-\boldsymbol{\xi}_{\sigma}$ 
respectively, i.\,e., put by definition
\begin{equation*}
\Phi_{\sigma}=\frac{d}{dt}\biggl|_{t=0}\varphi_{\sigma}(t),\quad
\boldsymbol{\xi}_{\sigma}=\frac{d}{dt}\biggl|_{t=0}\boldsymbol{x}_{\sigma}(t),\  \dots , \ 
\boldsymbol{W''}=\frac{d}{dt}\biggl|_{t=0}\boldsymbol{w''}(t)=\boldsymbol{\zeta}_{\sigma}''-\boldsymbol{\xi}_{\sigma}.
\end{equation*}

Note that 

$\bullet$ the length of the edge $\sigma(t)$ of $P_t$ is independent of $t$ and is equal to $\ell_{\sigma}$;

$\bullet$ the areas of the faces of $P_t$ with the vertices 
$\boldsymbol{x}_{\sigma}(t)$, $\boldsymbol{y}_{\sigma}(t)$, $\boldsymbol{z}_{\sigma}'(t)$ and
$\boldsymbol{x}_{\sigma}(t)$, $\boldsymbol{y}_{\sigma}(t)$, $\boldsymbol{z}_{\sigma}''(t)$ 
are also indepentent of $t$ and are equal to the areas of the faces of $P_0$ with the vertices
$\boldsymbol{x}_{\sigma}$, $\boldsymbol{y}_{\sigma}$, $\boldsymbol{z}_{\sigma}'$ and
$\boldsymbol{x}_{\sigma}$, $\boldsymbol{y}_{\sigma}$, $\boldsymbol{z}_{\sigma}''$ respectively
(equivalently, are equal to $\frac12 A'_{\sigma}$ and $\frac12 A''_{\sigma}$ respectively).

\begin{lemma}\label{l2}
Under the above notation, for every edge $\sigma$ of $P_0$, the equalities hold true
\begin{eqnarray}
(A'_{\sigma} A''_{\sigma}\sin\varphi_{\sigma})\Phi_{\sigma} & = &
\boldsymbol{p}_{\sigma}\cdot (\boldsymbol{\eta}_{\sigma}-\boldsymbol{\xi}_{\sigma})+ 
\boldsymbol{p}_{\sigma}'\cdot (\boldsymbol{\zeta}_{\sigma}'-\boldsymbol{\xi}_{\sigma}) +
\boldsymbol{p}_{\sigma}''\cdot (\boldsymbol{\zeta}_{\sigma}''-\boldsymbol{\xi}_{\sigma}), \label{eq4} \\
(A'_{\sigma}A''_{\sigma}\cos\varphi_{\sigma})\Phi_{\sigma} & = &
\boldsymbol{q}_{\sigma}\cdot (\boldsymbol{\eta}_{\sigma}-\boldsymbol{\xi}_{\sigma})+ 
\boldsymbol{q}_{\sigma}'\cdot (\boldsymbol{\zeta}_{\sigma}'-\boldsymbol{\xi}_{\sigma}) +
\boldsymbol{q}_{\sigma}''\cdot (\boldsymbol{\zeta}_{\sigma}''-\boldsymbol{\xi}_{\sigma}),  \label{eq5}
\end{eqnarray}
where
\begin{eqnarray*}
\boldsymbol{p}_{\sigma} & = & \bigl[(\boldsymbol{y}_{\sigma} - \boldsymbol{x}_{\sigma})\cdot 
(\boldsymbol{z}_{\sigma}'' - \boldsymbol{x}_{\sigma})\bigr](\boldsymbol{z}_{\sigma}' - \boldsymbol{x}_{\sigma})
+ \bigl[(\boldsymbol{y}_{\sigma} - \boldsymbol{x}_{\sigma})\cdot 
(\boldsymbol{z}_{\sigma}' - \boldsymbol{x}_{\sigma})\bigr](\boldsymbol{z}_{\sigma}'' - \boldsymbol{x}_{\sigma}), \\
\boldsymbol{p}_{\sigma}' & = & \bigl[(\boldsymbol{y}_{\sigma} - \boldsymbol{x}_{\sigma})\cdot 
(\boldsymbol{z}_{\sigma}'' - \boldsymbol{x}_{\sigma})\bigr](\boldsymbol{y}_{\sigma} - \boldsymbol{x}_{\sigma})
- \ell^2_{\sigma}(\boldsymbol{z}_{\sigma}'' - \boldsymbol{x}_{\sigma}), \\
\boldsymbol{p}_{\sigma}'' & = & \bigl[(\boldsymbol{y}_{\sigma} - \boldsymbol{x}_{\sigma})\cdot 
(\boldsymbol{z}_{\sigma}' - \boldsymbol{x}_{\sigma})\bigr](\boldsymbol{y}_{\sigma} - \boldsymbol{x}_{\sigma})
- \ell^2_{\sigma}(\boldsymbol{z}_{\sigma}' - \boldsymbol{x}_{\sigma}), \\
\boldsymbol{q}_{\sigma} & = & - \ell_{\sigma} (\boldsymbol{z}_{\sigma}' - \boldsymbol{x}_{\sigma})\times 
(\boldsymbol{z}_{\sigma}'' - \boldsymbol{x}_{\sigma}), \\
\boldsymbol{q}_{\sigma}' & = & \ell_{\sigma} (\boldsymbol{y}_{\sigma} - \boldsymbol{x}_{\sigma})\times 
(\boldsymbol{z}_{\sigma}'' - \boldsymbol{x}_{\sigma}), \\
\boldsymbol{q}_{\sigma}'' & = & - \ell_{\sigma} (\boldsymbol{y}_{\sigma} - \boldsymbol{x}_{\sigma})\times 
(\boldsymbol{z}_{\sigma}' - \boldsymbol{x}_{\sigma}).
\end{eqnarray*}
\end{lemma}

\begin{proof} 
Differentiating the formula (\ref{eq2}) with respect to $t$ when $t=0$ yields
\begin{eqnarray*}
&-&(A'_{\sigma}A''_{\sigma}\sin\varphi_{\sigma})\Phi_{\sigma}\\
& & = \ell^2_{\sigma} (\boldsymbol{W'}\cdot \boldsymbol{w''} + \boldsymbol{w''}\cdot \boldsymbol{W''}) 
- (\boldsymbol{W'}\cdot \boldsymbol{w} + \boldsymbol{w'}\cdot \boldsymbol{W})(\boldsymbol{w''}\cdot \boldsymbol{w}) \\
& & \hphantom{AAAAAAAAAAAAAAAll}
- (\boldsymbol{w'}\cdot \boldsymbol{w})(\boldsymbol{W''}\cdot \boldsymbol{w} + \boldsymbol{w''}\cdot \boldsymbol{W}) \\
& & = -\bigl[(\boldsymbol{w}\cdot \boldsymbol{w''})\boldsymbol{w'} + (\boldsymbol{w}\cdot \boldsymbol{w'})\boldsymbol{w''}
\bigr]\cdot \boldsymbol{W} - \bigl[(\boldsymbol{w}\cdot \boldsymbol{w''})\boldsymbol{w} 
- \ell^2_{\sigma}\boldsymbol{w''}\bigr]\cdot \boldsymbol{W'} \\
& & \hphantom{AAAAAAAAAAAAAAAAAAAAAl}
- \bigl[(\boldsymbol{w}\cdot \boldsymbol{w'})\boldsymbol{w} - \ell^2_{\sigma} \boldsymbol{w'}\bigr]\cdot \boldsymbol{W''}.
\end{eqnarray*}
This formula differs from (\ref{eq4}) by notation only.

Differentiating the formula (\ref{eq3}) with respect to $t$ when $t=0$ yields
\begin{eqnarray*}
(A'_{\sigma}A''_{\sigma}/\ell_{\sigma})(\cos\varphi_{\sigma})\Phi_{\sigma} & = &
- \boldsymbol{W}\cdot (\boldsymbol{w'}\times \boldsymbol{w''}) 
- \boldsymbol{w}\cdot (\boldsymbol{W'}\times \boldsymbol{w''}) - \boldsymbol{w}\cdot (\boldsymbol{w'}\times \boldsymbol{W''}) \\
& = & - (\boldsymbol{w'}\times \boldsymbol{w''})\cdot \boldsymbol{W} 
- (\boldsymbol{w''}\times \boldsymbol{w})\cdot \boldsymbol{W'} - (\boldsymbol{w}\times \boldsymbol{w'})\cdot \boldsymbol{W''}.
\end{eqnarray*}
This formula differs from (\ref{eq5}) by notation only.
\end{proof}

\begin{lemma}\label{l3}
Under the above notation, for every edge $\sigma$ of $P_0$, the equalities hold true
\begin{equation*}
\Phi_{\sigma}=
\begin{cases}
\boldsymbol{r}_{\sigma}\cdot (\boldsymbol{\eta}_{\sigma}-\boldsymbol{\xi}_{\sigma})
+ \boldsymbol{r}_{\sigma}'\cdot (\boldsymbol{\zeta}_{\sigma}'-\boldsymbol{\xi}_{\sigma}) 
+ \boldsymbol{r}_{\sigma}''\cdot (\boldsymbol{\zeta}_{\sigma}''-\boldsymbol{\xi}_{\sigma}), 
& \text{if \  $\sin\varphi_{\sigma}\neq 0$ }, \\
\boldsymbol{s}_{\sigma}\cdot (\boldsymbol{\eta}_{\sigma}-\boldsymbol{\xi}_{\sigma})
+ \boldsymbol{s}_{\sigma}'\cdot (\boldsymbol{\zeta}_{\sigma}'-\boldsymbol{\xi}_{\sigma}) 
+ \boldsymbol{s}_{\sigma}''\cdot (\boldsymbol{\zeta}_{\sigma}''-\boldsymbol{\xi}_{\sigma}), 
& \text{if \ $\cos\varphi_{\sigma}\neq 0$ }, 
\end{cases}
\end{equation*}
where
\begin{eqnarray*}
\boldsymbol{r}_{\sigma} & = & -\frac{\bigl[(\boldsymbol{y}_{\sigma} - \boldsymbol{x}_{\sigma})\cdot 
(\boldsymbol{z}_{\sigma}'' - \boldsymbol{x}_{\sigma})\bigr](\boldsymbol{z}_{\sigma}' - \boldsymbol{x}_{\sigma})
+ \bigl[(\boldsymbol{y}_{\sigma} - \boldsymbol{x}_{\sigma})\cdot (\boldsymbol{z}_{\sigma}' - \boldsymbol{x}_{\sigma})\bigr]
(\boldsymbol{z}_{\sigma}'' - \boldsymbol{x}_{\sigma})}{\ell_{\sigma}(\boldsymbol{y}_{\sigma}-
\boldsymbol{x}_{\sigma})\cdot \bigl[(\boldsymbol{z}_{\sigma}'-\boldsymbol{x}_{\sigma})\times 
(\boldsymbol{z}_{\sigma}''- \boldsymbol{x}_{\sigma}\bigr)]}, \\
\boldsymbol{r}_{\sigma}' & = & -\frac{\bigl[(\boldsymbol{y}_{\sigma} - \boldsymbol{x}_{\sigma})\cdot 
(\boldsymbol{z}_{\sigma}'' - \boldsymbol{x}_{\sigma})\bigr](\boldsymbol{y}_{\sigma} - \boldsymbol{x}_{\sigma})
- \ell^2_{\sigma}(\boldsymbol{z}_{\sigma}'' - \boldsymbol{x}_{\sigma})}{\ell_{\sigma}(\boldsymbol{y}_{\sigma}-
\boldsymbol{x}_{\sigma})\cdot \bigl[(\boldsymbol{z}_{\sigma}'-\boldsymbol{x}_{\sigma})\times 
(\boldsymbol{z}_{\sigma}''-\boldsymbol{x}_{\sigma})\bigr]}, \\
\boldsymbol{r}_{\sigma}'' & = & -\frac{\bigl[(\boldsymbol{y}_{\sigma} - \boldsymbol{x}_{\sigma})\cdot 
(\boldsymbol{z}_{\sigma}' - \boldsymbol{x}_{\sigma})\bigr](\boldsymbol{y}_{\sigma} - \boldsymbol{x}_{\sigma})
- \ell^2_{\sigma}(\boldsymbol{z}_{\sigma}' - \boldsymbol{x}_{\sigma})}{\ell_{\sigma}(\boldsymbol{y}_{\sigma}-
\boldsymbol{x}_{\sigma})\cdot \bigl[(\boldsymbol{z}_{\sigma}'-\boldsymbol{x}_{\sigma})\times 
(\boldsymbol{z}_{\sigma}''-\boldsymbol{x}_{\sigma})\bigr]}, \\
\boldsymbol{s}_{\sigma} & = & - \frac{\ell_{\sigma}}{\theta} (\boldsymbol{z}_{\sigma}' - \boldsymbol{x}_{\sigma})\times 
(\boldsymbol{z}_{\sigma}'' - \boldsymbol{x}_{\sigma}), \\
\boldsymbol{s}_{\sigma}' & = &  \frac{\ell_{\sigma}}{\theta}
(\boldsymbol{y}_{\sigma} - \boldsymbol{x}_{\sigma})\times (\boldsymbol{z}_{\sigma}'' - \boldsymbol{x}_{\sigma}), \\
\boldsymbol{s}_{\sigma}'' & = & - \frac{\ell_{\sigma}}{\theta} (\boldsymbol{y}_{\sigma} - \boldsymbol{x}_{\sigma})\times 
(\boldsymbol{z}_{\sigma}' - \boldsymbol{x}_{\sigma}),
\end{eqnarray*}
and
\begin{equation*}
\theta = \ell^2_{\sigma}(\boldsymbol{z}_{\sigma}'-\boldsymbol{x}_{\sigma})
\cdot(\boldsymbol{z}_{\sigma}''-\boldsymbol{x}_{\sigma})-\bigl[(\boldsymbol{z}_{\sigma}'-\boldsymbol{x}_{\sigma})
\cdot (\boldsymbol{y}_{\sigma}-\boldsymbol{x}_{\sigma})\bigr]
\bigl[(\boldsymbol{z}_{\sigma}''-\boldsymbol{x}_{\sigma})\cdot (\boldsymbol{y}_{\sigma}-\boldsymbol{x}_{\sigma})\bigr].
\end{equation*}
\end{lemma}

\begin{proof}
Lemma \ref{l3} follows immediately from Lemmas \ref{l1} and \ref{l2}.
\end{proof}

In \cite{GI18}, among other results, A.A. Gaifullin and L.S. Ignashchenko proved that the Dehn invariants of 
any flexible oriented polyhedron with non-degene\-rate triangular faces in $\mathbb{R}^3$ do not alter during a flex.
In \cite {Al19}, this fact was used to prove the following theorem

\begin{thrm}{\cite[Theorem 3]{Al19}}\label{thrm1}
Let $P_0$  be an oriented compact boundary-free flexible polyhedron in $\mathbb{R}^3$
with non-degenerate triangular faces and let $\{P_t\}_{t\in [0,1]}$ be a flex of $P_0$.
Let $\ell_{\sigma}$ be the length of an edge $\sigma$ of $P_0$ and let
$L$ be the $\mathbb{Q}$-linear span of the set $\cup_{\sigma\subset P_0} \{\ell_{\sigma}\}$ in $\mathbb{R}$.
Let $\{\lambda_1,\dots,\lambda_m\}$  be a $\mathbb{Q}$-basis in $L$,  i.\,e., the reals
$\lambda_1,\dots,\lambda_m$ are $\mathbb{Q}$-linearly independent and, for every $\sigma$,
there are $\alpha_{\sigma j}\in\mathbb{Q}$ such that
\begin{equation*}
\ell_{\sigma}=\sum_{j=1}^m \alpha_{\sigma j}\lambda_j.
\end{equation*}
Then, for every $j=1, \dots, m$, the expression 
\begin{equation}\label{eq6}
\sum_{\sigma\subset P_0}\alpha_{\sigma j} \varphi_{\sigma} (t)
\end{equation}
is independent of $t$, i.\,e., it remains unaltered during the flex $\{P_t\}_{t\in [0,1]}$.
In (\ref{eq6}), the summation is taken over all edges $\sigma$ of $P_0$ and 
$\varphi_{\sigma}(t)$ stands for the internal dihedral angle of the polyhedron $P_t$ at the edge $\sigma(t)$.
\end{thrm}

Observe that Theorem \ref{thrm1} implies the following statement:

\begin{thrm}\label{thrm2}
Under the hypothesis of Theorem \ref{thrm1}, for every $j=1, \dots, m$, the equality
\begin{equation}\label{eq7}
\sum_{\sigma\subset P_0}\alpha_{\sigma j}\bigl[\boldsymbol{g}_{\sigma}\cdot (\boldsymbol{\eta}_{\sigma}-\boldsymbol{\xi}_{\sigma})
+ \boldsymbol{g}_{\sigma}'\cdot (\boldsymbol{\zeta}_{\sigma}'-\boldsymbol{\xi}_{\sigma})
+ \boldsymbol{g}_{\sigma}''\cdot (\boldsymbol{\zeta}_{\sigma}''-\boldsymbol{\xi}_{\sigma})\bigr] = 0
\end{equation}
holds true, where
\begin{equation*}
\boldsymbol{g}_{\sigma}=
\begin{cases}
\boldsymbol{r}_{\sigma}, & \text{if \ $\cos\varphi_{\sigma}= 0$}; \\
\boldsymbol{s}_{\sigma}, & \text{if \ $\sin\varphi_{\sigma}= 0$}; \\
\text{either} \ \boldsymbol{r}_{\sigma} \ \text{or} \ \boldsymbol{s}_{\sigma}, & 
\text{if \ $(\cos\varphi_{\sigma})(\sin\varphi_{\sigma})\neq 0$}, 
\end{cases}
\end{equation*}
\begin{equation*}
\boldsymbol{g}_{\sigma}'=
\begin{cases}
\boldsymbol{r}_{\sigma}', & \text{if \ $\cos\varphi_{\sigma}= 0$}; \\
\boldsymbol{s}_{\sigma}', & \text{if \ $\sin\varphi_{\sigma}= 0$}; \\
\text{either} \ \boldsymbol{r}_{\sigma}' \ \text{or} \ \boldsymbol{s}_{\sigma}', & 
\text{if \ $(\cos\varphi_{\sigma})(\sin\varphi_{\sigma})\neq 0$}, 
\end{cases}
\end{equation*}
\begin{equation*}
\boldsymbol{g}_{\sigma}''= 
\begin{cases}
\boldsymbol{r}_{\sigma}'', & \text{if \ $\cos\varphi_{\sigma}= 0$}; \\
\boldsymbol{s}_{\sigma}'', & \text{if \ $\sin\varphi_{\sigma}= 0$}; \\
\text{either} \ \boldsymbol{r}_{\sigma}'' \ \text{or} \ \boldsymbol{s}_{\sigma}'', & 
\text{if \ $(\cos\varphi_{\sigma})(\sin\varphi_{\sigma})\neq 0$}
\end{cases}
\end{equation*}
and the vectors $\boldsymbol{r}_{\sigma}$, $\boldsymbol{r}_{\sigma}'$, $\boldsymbol{r}_{\sigma}''$, 
$\boldsymbol{s}_{\sigma}$, $\boldsymbol{s}_{\sigma}'$, and $\boldsymbol{s}_{\sigma}''$
are defined in Lemma \ref{l3}.
\end{thrm}

\begin{proof}
Differentiate the expression (\ref{eq6}) with respect to $t$ when $t=0$ and replace
$\frac{d}{dt}\bigl|_{t=0}\varphi_{\sigma} (t)=\Phi_{\sigma}$ by 
a suitable expression obtained in Lemma \ref{l3}.
\end{proof}

Generally speaking, for a given polyhedron $P_0$ we can write several equations (\ref{eq7}), not a single one.
There are several reasons for that.
Firstly, for each $j=1,\dots, m$, we can write a specific equation (\ref{eq7}).
Secondly, for each edge $\sigma$, such that $(\cos\varphi_{\sigma})(\sin\varphi_{\sigma})\neq 0$, 
there are two ways to write each of the coefficients 
$\boldsymbol{g}_{\sigma}$, $\boldsymbol{g}_{\sigma}'$, and $\boldsymbol{g}_{\sigma}''$.
That is, each edge of this type generates eight equations (\ref{eq7}).
In general, these eight equations are independent from each other, but the velocity vectors of
the vertices of a flexible polyhedron $P_0$ must satisfy each of them.

\section{Necessary conditions for the extendibility of a first-order flex of a polyhedron to its flex,
generated by the Dehn invariants}\label{s3}

Let $P$ be an arbitrary polyhedron in $\mathbb{R}^3$, let $\sigma$ be its edge,  
and let $\boldsymbol{x}_{\sigma}$, $\boldsymbol{y}_{\sigma}$, $\boldsymbol{z}_{\sigma}'$, $\boldsymbol{z}_{\sigma}''$ 
be the vertices of $P$, associated with $\sigma$ as described in Section \ref{s2}.
Let us assign arbitrary vectors $\boldsymbol{a}_{\sigma}$, $\boldsymbol{b}_{\sigma}$, $\boldsymbol{c}_{\sigma}'$, and 
$\boldsymbol{c}_{\sigma}''$ 
to the vertices $\boldsymbol{x}_{\sigma}$, $\boldsymbol{y}_{\sigma}$, $\boldsymbol{z}_{\sigma}'$, and 
$\boldsymbol{z}_{\sigma}''$, respectively. 
Note that this notation is overdetermined in the following sense:
a vector, denoted as $\boldsymbol{a}_{\sigma}$, $\boldsymbol{b}_{\sigma}$, $\boldsymbol{c}_{\sigma}'$, or 
$\boldsymbol{c}_{\sigma}''$ for the edge $\sigma$, can get any of the notation
$\boldsymbol{a}_{\tau}$, $\boldsymbol{b}_{\tau}$, $\boldsymbol{c}_{\tau}'$, or 
$\boldsymbol{c}_{\tau}''$ for another edge $\tau$.
Suppose that $P$ is deformed in such a way that the verticies 
$\boldsymbol{x}_{\sigma}$, $\boldsymbol{y}_{\sigma}$, $\boldsymbol{z}_{\sigma}'$, $\boldsymbol{z}_{\sigma}''$ 
move according to the formulas
\begin{equation}\label{eq8}
\boldsymbol{x}_{\sigma}+t\boldsymbol{a}_{\sigma},\quad
\boldsymbol{y}_{\sigma}+t\boldsymbol{b}_{\sigma},\quad
\boldsymbol{z}_{\sigma}'+t\boldsymbol{c}_{\sigma}',\quad
\boldsymbol{z}_{\sigma}''+t\boldsymbol{c}_{\sigma}'', 
\end{equation} 
where $t\in\mathbb{R}$.
In Section \ref{s1}, we called this deformation a first-order flex of $P$
if it changes the length of every edge of $P$ by a value which is $o(t)$ as $t\to 0$.
From now we say that, the set of vectors 
$\cup_{\sigma\subset P}\{\boldsymbol{a}_{\sigma}, \boldsymbol{b}_{\sigma}, 
\boldsymbol{c}_{\sigma}', \boldsymbol{c}_{\sigma}''\}$ is a \textit{first-order flex} of $P$ 
if the deformation (\ref{eq8}) changes the length of every edge of $P$ by a value which is $o(t)$ as $t\to 0$.
Such a switch from deformations to collections of vectors is common in rigidity theory, see, e.\,g.,
\cite{Co93} or \cite{IS95}. 

The condition that the length of $\sigma$ is stationary at $t=0$ for the  first-order flex
$\cup_{\sigma\subset P}\{\boldsymbol{a}_{\sigma}, \boldsymbol{b}_{\sigma}, 
\boldsymbol{c}_{\sigma}', \boldsymbol{c}_{\sigma}''\}$ yields
\begin{eqnarray*}
0 & = & \frac{d}{dt}\biggl|_{t=0}\bigl[(\boldsymbol{y}_{\sigma}+t\boldsymbol{b}_{\sigma}
-\boldsymbol{x}_{\sigma}-t\boldsymbol{a}_{\sigma})\cdot
(\boldsymbol{y}_{\sigma}+t\boldsymbol{b}_{\sigma}-\boldsymbol{x}_{\sigma}-t\boldsymbol{a}_{\sigma})\bigr]\\
& = & \bigl[2(\boldsymbol{y}_{\sigma}-\boldsymbol{x}_{\sigma})\cdot (\boldsymbol{b}_{\sigma}-\boldsymbol{a}_{\sigma}) + 
2t(\boldsymbol{b}_{\sigma}-\boldsymbol{a}_{\sigma})\cdot (\boldsymbol{b}_{\sigma}-\boldsymbol{a}_{\sigma})
\bigr]\bigl|_{t=0}\\
& = & 2(\boldsymbol{y}_{\sigma}-\boldsymbol{x}_{\sigma})\cdot (\boldsymbol{b}_{\sigma}-\boldsymbol{a}_{\sigma}),
\end{eqnarray*}
i.\,e.,
\begin{equation}\label{eq9}
(\boldsymbol{y}_{\sigma}-\boldsymbol{x}_{\sigma})\cdot (\boldsymbol{b}_{\sigma}-\boldsymbol{a}_{\sigma})=0.
\end{equation}
The equation (\ref{eq9}), corresponding to an edge $\sigma$ of $P$,
is called an \textit{equation of the first-order flexes} of $P$.

Denote by $\langle FOF\rangle$ the set of all finite $\mathbb{R}$-linear combinations of the equations (\ref{eq9}). 
This notation reflects the fact that this set is generated by the equations of the first-order flexes.

We say that a first-order flex 
$\cup_{\sigma\subset P}\{\boldsymbol{a}_{\sigma}, \boldsymbol{b}_{\sigma}, 
\boldsymbol{c}_{\sigma}', \boldsymbol{c}_{\sigma}''\}$ of $P$ 
\textit{can be extended to a flex} of $P$ if there is a flex $\{ P_t\}_{t\in [0,1]}$ of 
$P=P_0$, such that if the functions
$\boldsymbol{x}_{\sigma}(t)$, $\boldsymbol{y}_{\sigma}(t)$, $\boldsymbol{z}_{\sigma}'(t)$, $\boldsymbol{z}_{\sigma}''(t)$
represent the motions of the vertices of $P_t$ corresponding to the vertices 
$\boldsymbol{x}_{\sigma}$, $\boldsymbol{y}_{\sigma}$, $\boldsymbol{z}_{\sigma}'$, $\boldsymbol{z}_{\sigma}''$ 
of $P_0$ then
\begin{equation*}
\boldsymbol{a}_{\sigma}=\frac{d}{dt}\biggl|_{t=0}\boldsymbol{x}_{\sigma}(t), \biggr. \ \ 
\boldsymbol{b}_{\sigma}=\frac{d}{dt}\biggl|_{t=0}\boldsymbol{y}_{\sigma}(t), \biggr. \ \ 
\boldsymbol{c}_{\sigma}'=\frac{d}{dt}\biggl|_{t=0} \boldsymbol{z}_{\sigma}'(t) \biggr. , \ \ 
\boldsymbol{c}_{\sigma}''=\frac{d}{dt}\biggl|_{t=0} \boldsymbol{z}_{\sigma}''(t). \biggr.
\end{equation*}

Let us observe that Theorem \ref{thrm2} implies the following statement:

\begin{thrm}\label{thrm3}
Let $P$  be an oriented compact boundary-free flexible polyhedron in $\mathbb{R}^3$
with non-degenerate triangular faces.
Let $\ell_{\sigma}$ be the length of an edge $\sigma$ of $P$ and let
$L$ be the $\mathbb{Q}$-linear span of the set $\cup_{\sigma\subset P} \{\ell_{\sigma}\}$ in $\mathbb{R}$.
Let $\{\lambda_1,\dots,\lambda_m\}$  be a $\mathbb{Q}$-basis in $L$,  i.\,e., the reals
$\lambda_1,\dots,\lambda_m$ are $\mathbb{Q}$-linearly independent and, for every $\sigma$,
there are $\alpha_{\sigma j}\in\mathbb{Q}$ such that
\begin{equation*}
\ell_{\sigma}=\sum_{j=1}^m \alpha_{\sigma j}\lambda_j.
\end{equation*}
Let $\cup_{\sigma\subset P}\{\boldsymbol{a}_{\sigma}, \boldsymbol{b}_{\sigma}, 
\boldsymbol{c}_{\sigma}', \boldsymbol{c}_{\sigma}''\}$ be a first-order flex of $P$
which can be extended to a flex of $P$. 
Then, for every $j=1, \dots, m$, the equality
\begin{equation}\label{eq10}
\sum_{\sigma\subset P}\alpha_{\sigma j}\bigl[\boldsymbol{g}_{\sigma}\cdot (\boldsymbol{b}_{\sigma}-\boldsymbol{a}_{\sigma})
+ \boldsymbol{g}_{\sigma}'\cdot (\boldsymbol{c}_{\sigma}'-\boldsymbol{a}_{\sigma})
+ \boldsymbol{g}_{\sigma}''\cdot (\boldsymbol{c}_{\sigma}''-\boldsymbol{a}_{\sigma})\bigr] = 0
\end{equation}
holds true, where the vectors
$\boldsymbol{g}_{\sigma}$, $\boldsymbol{g}_{\sigma}'$, $\boldsymbol{g}_{\sigma}''$
were defined in Theorem \ref{thrm2}.
\end{thrm}

\begin{proof}
Immediately follows from Theorem \ref{thrm2} and from the definition of a first-order flex 
which can be extended to a flex.
\end{proof}

Denote by $\langle DI\rangle$ the set of all finite $\mathbb{R}$-linear combinations of the equations (\ref{eq10}). 
This notation reflects the fact that this set is generated by the Dehn invariants.

In our opinion, the set $\langle DI\rangle$ deserves deep study.
The following problems are especially interesting:

(i) How many new (compared to (\ref{eq9})) equations are contained among the equations (\ref{eq10}), 
i.\,e., how much the dimension of the linear span of the set $\langle DI\rangle\cup\langle FOF\rangle$ 
is greater than  the dimension of the linear space $\langle FOF\rangle$?

(ii) Characterize all $n\geqslant 2$ such that the equations (\ref{eq10}) are not a consequence of the 
equations of the $n$th order flex.

In \cite {St99}, Hellmuth Stachel constructed  explicit examples of octahedra in $\mathbb{R}^3$ which
admit the first- and second-order flexes as well as explicit necessary and
sufficient conditions for an octahedron to be $n$th-order infinitesimally flexible for $n<8$,
provided the octahedron under consideration is not totally flat.
These results may be useful in the study of the problems (i) and (ii). 
But we left such a study for the future.

In Example 1 below, we answer the simplest problem similar to the problems (i) and (ii).
Namely, we show that, there are a polyhedron and its edge $\sigma$ such that, the equation (\ref{eq10}),
written for $\sigma$, is linearly independent of all equations (\ref{eq9}).

\textbf{Example 1}. Let $Q$ be a polyhedron in $\mathbb{R}^3$ shown in Fig.~\ref{fig2}.
\begin{figure}
\begin{center}
\includegraphics[width=0.2\textwidth]{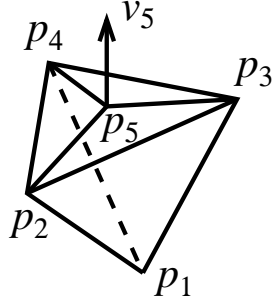}
\end{center}
\caption{Polyhedron $Q$ and the vector $\boldsymbol{v}_5$ which is the only non-zero vector among the vectors 
$\boldsymbol{v}_i$, $i=1,2,3,4,5$.}\label{fig2}
\end{figure}
In other words, let $Q$ be a triangular bipyramid with vertices $\boldsymbol{p}_i$, $i=1,\dots, 5$.
Suppose that $\boldsymbol{p}_1$, $\boldsymbol{p}_2$, $\boldsymbol{p}_3$, $\boldsymbol{p}_4$
form a regular tetrahedron  with edge length 8 and $\boldsymbol{p}_5$ lies inside the face 
$\boldsymbol{p}_2$, $\boldsymbol{p}_3$, $\boldsymbol{p}_4$ in such a manner that the length of each of the 
two edges with the endpoints $\boldsymbol{p}_3$, $\boldsymbol{p}_5$ and $\boldsymbol{p}_4$, $\boldsymbol{p}_5$ 
is equal to 5.
For convenience, we assume that 
\begin{multline*}
\boldsymbol{p}_1 =  \biggl(\frac{4}{3}\sqrt{3}-3, 0, -\frac{8}{3}\sqrt{6}\biggr),\quad
\boldsymbol{p}_2 =  (4\sqrt{3}-3, 0, 0),\\
\boldsymbol{p}_3 =  (-3, 4, 0),\quad
\boldsymbol{p}_4 =  (-3,-4, 0),\quad
\boldsymbol{p}_5 =  (0, 0, 0).
\end{multline*}

Let us associate the velocity vector $\boldsymbol{v}_i$ to each of the vectors  $\boldsymbol{p}_i$, $i=1,\dots,5$,
and put by definition $\boldsymbol{v}_5 = (0, 0, 1)$ and $\boldsymbol{v}_i=(0,0,0)$ for every  $i=1,2,3,4$.

Since $\boldsymbol{v}_5$ is orthogonal to the plane containing the vertices $\boldsymbol{p}_i$, $i=2,3,4,5$,
then $(\boldsymbol{p}_i-\boldsymbol{p}_5)\cdot\boldsymbol{v}_5=0$ for every $i=2,3,4$. 
These formulas differ from (\ref{eq9}) by notation only.
Hence, the vectors $\boldsymbol{v}_i$, $i=1,2,3,4,5$, constitute the first-order flex of $Q$.

Now let us verify that the vectors $\boldsymbol{v}_i$, $i=1,2,3,4,5$, do not satisfy the equation (\ref{eq10}), 
written for the edge with the end-points $\boldsymbol{p}_2$, $\boldsymbol{p}_5$.
Denote this egde by $\sigma$. Then
$\boldsymbol{x}_{\sigma}=\boldsymbol{p}_2$, $\boldsymbol{y}_{\sigma}=\boldsymbol{p}_5$, 
$\boldsymbol{z}_{\sigma}'=\boldsymbol{p}_4$, $\boldsymbol{z}_{\sigma}''=\boldsymbol{p}_3$, 
$\boldsymbol{n}_{\sigma}'=\boldsymbol{n}_{\sigma}''=(0,0,1)$, 
$\boldsymbol{\xi}_{\sigma}=\boldsymbol{\zeta}_{\sigma}'=\boldsymbol{\zeta}_{\sigma}''=(0,0,0)$,
$\boldsymbol{\eta}_{\sigma}=\boldsymbol{v}_5=(0,0,1)$, $\ell_{\sigma}=4\sqrt{3}-3$, $\ell'_{\sigma}=\ell''_{\sigma}=8$.

By definition, put $\lambda_1=1$ and $\lambda_2=4\sqrt{3}-3$.
Since the length of any edge of $Q$ is equal to 5, 8, or $4\sqrt{3}-3$, we conclude that the set
$\{\lambda_1, \lambda_2\}$ is a $\mathbb{Q}$-basis in the $\mathbb{Q}$-linear span of the set of all edge-lengths of $Q$.
Hence, $\alpha_{\sigma 1}=0$, $\alpha_{\sigma 2}=1$ and the equation  (\ref{eq10}) takes the form
\begin{equation}\label{eq11}
\boldsymbol{g}_{\sigma}\cdot\boldsymbol{v}_5 =0,
\end{equation}
where $\boldsymbol{g}_{\sigma} = \boldsymbol{s}_{\sigma}$ since $\sin\varphi_{\sigma}=0$.
Calculating $\boldsymbol{s}_{\sigma}$ according to a suitable formula from Lemma \ref{l3}, we get
\begin{equation*}
\boldsymbol{g}_{\sigma}= \boldsymbol{s}_{\sigma}=\biggl(0,0,\frac{128-6\sqrt{3}}{313} \biggr).
\end{equation*}

Since $\boldsymbol{v}_5=(0,0,1)$, we conclude that the equality (\ref{eq11}) is not satisfied.
Thus, the equation (\ref{eq10}) corresponding to the above chosen edge $\sigma$ is not satisfied.

Hence, for $Q$, one of the equations (\ref{eq10}) is not a linear combination of the equations 
(\ref{eq9}). This completes Example 1.

\section{Necessary conditions for the extendibility of a first-order flex of a polyhedron to its flex,
generated by the rigidity matrix}\label{s4}

In Sections \ref{s2} and \ref{s3} we have shown that the Dehn invariants imply the equations (\ref{eq10}) 
which are necessary conditions for a given first-order flex of a polyhedron to be extendable to a flex.
In this Section we show that the rigidity matrix also implies a finite set of linearly independent equations
which are necessary conditions for a given first-order flex of a polyhedron to be extendable to a flex.
 
First we fix the notation which will differ slightly from the notation used in Sections \ref{s2} and \ref{s3}.

Let $P_0$ be a flexible polyhedron in $\mathbb{R}^3$ with triangular faces. 
(Note that in this Section we do not exclude the cases when $P_0$ is non-orientable or 
some of its faces are degenerated.) Suppose $P_0$ has $V$ vertices.
Let us enumerate them in an arbitrary order and denote by 
$\boldsymbol{p}_i=(x_i, y_i, z_i)\in\mathbb{R}^3$, $i=1,\dots,V$.
Let $\{ P_t \}_{t\in [0,1]}$ be a flex of $P_0$.
Generally speaking, the vertices of the polyhedron $P_t$ are functions of $t$.
Denote them by $\boldsymbol{p}_i(t)=\bigl(x_i(t), y_i(t), z_i(t)\bigr)$, $i=1,\dots,V$.
Moreover, we will assume that $\boldsymbol{p}_i(0)=\boldsymbol{p}_i$.

Let $\boldsymbol{v}_i(t)=\bigl(v_{i,1}(t), v_{i,2}(t), v_{i,3}(t)\bigr)\in\mathbb{R}^3$ 
be the velocity vector of the vertex $\boldsymbol{v}_i(t)$, i.\,e., put by definition
\begin{equation}\label{eq12}
\boldsymbol{v}_i(t)=\frac{d}{dt}\boldsymbol{p}_i(t), \ \ 
v_{i,1}(t)=\frac{d}{dt}x_i(t), \ \ 
v_{i,2}(t)=\frac{d}{dt}y_i(t), \ \ 
v_{i,3}(t)=\frac{d}{dt}z_i(t).
\end{equation}

If the vertices $\boldsymbol{v}_i$ and $\boldsymbol{v}_j$ of $P_0$ are joint with each other by an edge
then the length of the corresponding (by continuity) edge of $P_t$ is independent of $t$, namely,
\begin{eqnarray*}
\bigl(x_i(t)-x_j(t)\bigr)^2  +  \bigl(y_i(t)-y_j(t)\bigr)^2 & + & \bigl(z_i(t)-z_j(t)\bigr)^2 \\
& = & (x_i-x_j)^2 + (y_i-y_j)^2 + (z_i-z_j)^2.
\end{eqnarray*}
Differentiating the last relation in $t$ we get
\begin{multline}\label{eq13}
\bigl(x_i(t)-x_j(t)\bigr)\bigl(v_{i,1}(t)-v_{j,1}(t)\bigr) 
+ \bigl(y_i(t)-y_j(t)\bigr)\bigl(v_{i,2}(t)-v_{j,2}(t)\bigr) \\
+ \bigl(z_i(t)-z_j(t)\bigl)\bigr(v_{i,3}(t)-v_{j,3}(t)\bigr) = 0.
\end{multline}

For a fixed $t$, the equations (\ref{eq13}) are linear with respect to $3V$ variables $v_{i,k}(t)$, $i=1,\dots,V$, $k=1,2,3$.
Jointly, the equations (\ref{eq13}) form a homogeneous system of linear algebraic equations.
The matrix of this system is called the \textit{rigidity matrix} of the polyhedron $P_t$.
It has $3V$ columns and $E$ rows, where $E$ stands for the number of edges of $P_0$.

\begin{lemma}\label{l4}
Let $P_0$ be a flexible polyhedron in $\mathbb{R}^3$ with triangular faces and $V$ vertices
and let $\{ P_t \}_{t\in [0,1]}$ be a flex of $P_0$.
Then, for every $t\in [0,1]$, every $(3V-7)\times (3V-7)$ minor of the rigidity matrix of 
the polyhedron $P_t$ is equal to zero.
\end{lemma}

\begin{proof}
Denote by $F$ the number of faces of $P_t$. According to Euler's formula, 
\begin{equation}\label{eq14}
V-E+F=\chi
=\begin{cases}
2-2g, & \text{if \  $P_t$ \ is orientable;} \\
 2-k, & \text{if \ $P_t$ \ is non-orientable.}
\end{cases}
\end{equation} 
Here $\chi$ is the Euler characteristic of $P_t$, $g\geqslant 0$ is genus of $P_t$ (i.\,e.,
the number of tori in a connected sum decomposition of $P_t$),  and $k\geqslant 0$ is non-orientable genus of $P_t$
(i.\,e., the number of real projective planes in a connected sum decomposition of $P_t$).
Since $P_t$ has triangular faces only,  $3F=2E$.
Substituting the latter relation to (\ref{eq14}), we obtain $E=3V-3\chi\geqslant 3V-6$. 
This inequality implies that the set of $(3V-6)\times (3V-6)$ minors of the rigidity matrix of $P_t$ is non-empty.

Since the isometry group of $\mathbb{R}^3$ has dimension 6, the kernel of the rigidity matrix of 
any polyhedron with triangular faces has dimension greater than or equal to 6.
Therefore, every $(3V-6)\times (3V-6)$ minor of the rigidity matrix of $P_t$ is equal to zero.

Recall that 

(a) the first-order flex of a polyhedron is called \textit{nontrivial} if it is not
generated by a continuous family of isometries of $\mathbb{R}^3$ or, equivalently,
if there are two vertices of the polyhedron (which are not connected by an edge) 
such that the Euclidean distance between these vertices is not stationary;

(b) a polyhedron is called \textit{infinitesimally flexible}, if it admits a nontrivial first-order flex;

(c) a polyhedron in $\mathbb{R}^3$ with triangular faces and $V$ vertices is infinitesimally flexible 
if and only if every $(3V-7)\times (3V-7)$ minor of its rigidity matrix is equal to zero;

(d) every flexible polyhedron is infinitesimally flexible.

These definitions and statements are standard in rigidity theory.
The reader can find them, for example, in \cite[Theorem 4.1]{Gl75} (for the simplest case of sphere-homeomorphic polyhedra) 
or in \cite[Theorem 3.1]{Co93} and \cite{CW96} (for a more general case of tensegrity frameworks).

Now we a ready to finalize the proof of Lemma \ref{l4}.
Obviously, $P_t$ is flexible for every $t\in [0,1]$. According to (d), it is infinitesimally flexible.
Then, (c) implies that every $(3V-7)\times (3V-7)$ minor of the rigidity matrix of $P_t$ is equal to zero.
\end{proof}

As before, suppose that $P_0$ is a flexible polyhedron in $\mathbb{R}^3$ with triangular faces and $V$ vertices,
$\{ P_t \}_{t\in [0,1]}$ is a flex of $P_0$, and $x_1(t)$, $y_1(t)$, $z_1(t)$, \dots, $x_V(t)$, $y_V(t)$, $z_V(t)$
stand for the coordinates of the vertices of $P_t$.
Then, for every $t\in [0,1]$, every  $(3V-7)\times (3V-7)$ minor $\Delta(t)$ of the rigidity matrix of $P_t$ is 
a homogeneous polynomail in the variables
$x_1(t)$, $y_1(t)$, $z_1(t)$, \dots, $x_V(t)$, $y_V(t)$, $z_V(t)$.
Moreover, according to Lemma \ref{l4}, this polynomial is equal to zero if the values of these variables 
correspond to $P_t$.
Differentiating $\Delta(t)$ with respect to $t$ when $t=0$ and replacing the derivatives of the variables
$x_1(t)$, $y_1(t)$, $z_1(t)$, \dots, $x_V(t)$, $y_V(t)$, $z_V(t)$ 
according to the formulas (\ref{eq12}), yields that some finite linear combination of the vectors
$\boldsymbol{v}_1(0)$, $\boldsymbol{v}_2(0)$, \dots, $\boldsymbol{v}_V(0)$
is equal to zero.
The coefficients of this linear combination are homogeneous polynomials in the variables
$x_1(0)$, $y_1(0)$, $z_1(0)$, \dots, $x_V(0)$, $y_V(0)$, $z_V(0)$.
Each such linear combination is a new (compared with the stationarity condition of the edge lengths)
relation that should be satisfied by every first-order flex
$\boldsymbol{v}_1(0)$, $\boldsymbol{v}_2(0)$, \dots, $\boldsymbol{v}_V(0)$ of $P_0$ 
which can be extended to a flex $\{ P_t \}_{t\in [0,1]}$.

Let us list a few problems related to the above constructed new linear relations for the first-order flexes:

(A) Is it true that the stationarity of all $(3V-7)\times (3V-7)$ minors of the rigidity matrix is a 
consequence of the stationarity of the edge lengths?

(B) How many new linearly independent relations are generated by $(3V-7)\times (3V-7)$ minors of 
the rigidity matrix of a given infinitesimally flexible polyhedron?

(C) How these new relations are related to the possibility of extension of a given first-order flex
to an $n$th order flex for $n\geqslant 2$?

(D) Do these new relations imply that the volume bounded by a flexible polyhedorn is stationary?
(It is known that the stationarity of the edge lengths only does not imply the stationarity of the
volume, see, e.\,g., \cite{Al10}).

In Example 2, we make the first step in the study of the problems (A)--(D).
Namely, we give a negative answer to the problem (A).

\textbf{Example 2}. Let $Q$ be the polyhedron with $V=5$ vertices constructed in Example 1.
Let us make sure that at least one linear relation for the first-order flexes obtained by differentiating 
the $(3V-7)\times (3V-7)$ minors of the rigidity matrix of $Q$ is not a consequence of
the equations expressing the stationarity of the edge lengths of $Q$.

To do this, we consider a first-order flex $\cup_{i=1}^5\{\boldsymbol{v}_i\}$ of $Q$
such that  $\boldsymbol{v}_i=0$ for every $i=1,\dots, 4$ and describe all vectors
$\boldsymbol{v}_5$ such that all $(3V-7)\times (3V-7)$ minors of the rigidity matrix of $Q$ are stationary 
at $t=0$ for the deformation 
\begin{equation}\label{eq15}
t\mapsto \boldsymbol{p}_i+t\boldsymbol{v}_i, \quad t\in\mathbb{R}, \quad i=1,\dots, 5.
\end{equation}

Obviously, the vector $\boldsymbol{v}_5$ must lie in the plane passing through the vertices
$\boldsymbol{p}_2$, $\boldsymbol{p}_3$, $\boldsymbol{p}_4$ of $Q$.
This follows from the fact that every polyhedron $Q'$ with the vertices $\boldsymbol{p}_i'$, $i=1,\dots,5$, 
which has the same combinatorial structure as $Q$ and is sufficiently close to $Q$
(i.\,e., such that the length of each of the vectors $\boldsymbol{p}_i'-\boldsymbol{p}_i$, $i=1,\dots,5$,
does not exceed a sufficiently small number $\varepsilon>0$) 
is infinitezimally rigid if and only if the vertex $\boldsymbol{p}_5'$ lies in the plane passing through the vertices
$\boldsymbol{p}_2'$, $\boldsymbol{p}_3'$, $\boldsymbol{p}_4'$.

Since $\cup_{i=1}^5\{\boldsymbol{v}_i\}$ is a first-order flex of $Q$, 
$\boldsymbol{v}_5$ is orthogonal to the plane passing through the vertices $\boldsymbol{p}_2$, $\boldsymbol{p}_3$, 
$\boldsymbol{p}_4$ of $Q$. 
For $\boldsymbol{v}_5\neq 0$, this implies that $\boldsymbol{v}_5$ does not lie in this plane
and, thus, some of the $(3V-7)\times (3V-7)$ minors of the rigidity matrix of $Q$ is not stationary under the deformation
(\ref{eq15}) at $t=0$.
  
\bibliographystyle{plain}
\bibliography{alex_nce_bib}
\end{document}